\documentclass[12pt]{article}
\usepackage{amsthm,amsmath,amssymb,amsfonts}
\usepackage{graphicx}
\usepackage[colorlinks=true,citecolor=black,linkcolor=black,urlcolor=blue]{hyperref}

\theoremstyle{plain}
\newtheorem{theorem}{Theorem}[section]

\newtheorem{lemma}[theorem]{Lemma}
\newtheorem{corollary}[theorem]{Corollary}
\newtheorem{proposition}[theorem]{Proposition}

\theoremstyle{definition}
\newtheorem{definition}[theorem]{Definition}
\newtheorem{example}[theorem]{Example}

\theoremstyle{remark}

\title{\bf Between graphical zonotope and graph\,--\,associahedron}

\author{
Marko Pe\v{s}ovi\'c\footnote{The author is supported by Ministry of Science of Republic of Serbia, Project 174034.}\\
\small Faculty of Civil Engineering,\\[-0.8ex]
\small University of Belgrade\\[-0.8ex]
\small Serbia\\
\small\tt mpesovic@grf.bg.ac.rs
\and
Tanja Stojadinovi\'c\footnote{The author is supported by the Science Fund of the Republic of Serbia, Grant No. 7749891, Graphical Languages -- GWORDS.
}\\
\small Faculty of Mathematics\\[-0.8ex]
\small University of Belgrade\\[-0.8ex]
\small Serbia\\
\small\tt tanjas@matf.bg.ac.rs}

\date{
}

\begin{document}

\maketitle

\begin{abstract}
This manuscript introduces a finite collection of generalized
permutohedra associated to a simple graph. The first polytope of this
collection is the graphical zonotope of the graph and the last is the
graph--associahedron associated to it. We describe the weighted integer
points enumerators for polytopes in this collection as Hopf algebra
morphisms of combinatorial Hopf algebras of decorated graphs. In the last
section, we study some properties related to $\mathcal{H}$--polytopes.

\end{abstract}

\section{Introduction }

\medskip

In this paper we construct a finite collection of generalized permutohedra associated to a simple graph. This collection starts with the graphical zonotope and ends with the graph--associahedron. To each generalized permutohedron $Q$ there is associated quasisymmetric function $F(Q)$, introduced in~\cite{BJR}. It enumerates positive lattice points in the normal fan $\mathcal{N}_Q$. The weighted analogue $F_q(Q)$ of this enumerator, which takes into account the face structure of the normal fan $\mathcal{N}_Q$, is introduced and studied in~\cite{GPS}. Among its properties is that the $f-$polynomial $f(Q)$ can be obtained as the principal specialization of $F_q(Q)$. For different classes of generalized permutohedra the algebraic interpretation of these enumerators is given as universal morphism from appropriately defined combinatorial Hopf algebras to the combinatorial Hopf algebra of quasisymmetric functions $\mathcal{Q}Sym$.

\medskip

The cases of graphical zonotopes $Z_\Gamma$ and graph--associahedra $P_\Gamma$ are of special interest. The enumerator $F(Z_\Gamma)$ is known to be the Stanley chromatic symmetric function and the enumerator $F(P_\Gamma)$ is the chromatic quasisymmetric function introduced in~\cite{G}. Our construction produces a finite collection of weighted quasisymmetric functions between $F_q(Z_\Gamma)$ and $F_q(P_\Gamma)$. We show that each of these weighted quasisymmetric functions is actually derived from certain combinatorial Hopf algebra of decorated graphs.

\section{Graph polytopes}

\medskip

For the standard basis vectors $\{e_s\}_{1\leq s\leq n}$ in $\mathbb{R}^n$, let $\Delta_H:=\mathsf{conv}\{e_s\,:\,s\in H\}$ be the simplex determined by a subset $H\subseteq[n]$. The \emph{hypergraphic polytope} of a hypergraph $\mathsf{H}$ on $[n]$ is the Minkowski sum of simplices
$$Q_\mathsf{H}\,:=\,\sum_{H\in\mathsf{H}}\Delta_H.$$
As generalized permutohedron can be described as the Minkowski sum of dilated simplices (see~\cite{P}), we have that hypergraphic polytope is generalized permutohedron, i.e. a convex polytope whose normal fan $\mathcal{N}_{Q_\mathsf{H}}$ is refined by the reduced normal fan $\mathcal{N}_{Pe^{n-1}}$ of standard $(n-1)-$dimensional permutohedron $Pe^{n-1}$. The $(n-d)-$dimensional faces of $Pe^{n-1}$ are in one$-$to$-$one correspondence with set compositions $\mathcal{C}=C_1|C_2|\cdots|C_{d}$ of the set $[n]$ (see~\cite{P}, Proposition 2.6). By this correspondence and the correspondence between set compositions and flags of subsets we identify face $\mathcal{C}$ of $Pe^{n-1}$ with the flag $\mathcal{F}:\emptyset=F_0\subset F_1\cdots\subset F_{d}=[n]$, where $F_m=\cup_{i=0}^mC_i$ for $1\leq m\leq d.$\\

Next, for a flag $\mathcal{F}$ of subsets let $M_\mathcal{F}$ be the enumerator of positive integer points $\omega=(\omega_1,\omega_2,\ldots,\omega_n)\in\mathbb{Z}^n_+$ in the interior of the normal cone $\mathcal{N}_{Pe^{n-1}}(\mathcal{F})$ at the $(n-d)-$dimensional face $\mathcal{F}$,
$$M_\mathcal{F}\;\;:=\;\sum_{\omega\in\mathbb{Z}^n_+\cap\,\mathcal{N}^\circ_{Pe^{n-1}}(\mathcal{F})}x_{\omega_1}x_{\omega_2}\cdots x_{\omega_n}.$$
The enumerator $M_\mathcal{F}$ is a monomial quasisymmetric function depending only on the composition $\mathsf{type}(\mathcal{F}):=(|F_1|,|F_2|-|F_1|,\ldots,|F_{d}|-|F_{d-1}|).$\\ 

Further, for a hypergraph $\mathsf{H}$, in ~\cite{MP} is defined its \emph{splitting hypergraph} $\mathsf{H}/\mathcal{F}$ by a flag $\mathcal{F}:F_0\subset F_1\subset\cdots\subset F_k$ with
$$\mathsf{H}/\mathcal{F}\;:=\;\bigsqcup_{i=1}^k(\mathsf{H}|_{F_i})/F_{i-1},$$
where the \emph{restriction} $\mathsf{H}|_F$ and the \emph{contraction}
$\mathsf{H}/F$ are defined by $\mathsf{H}|_F:=\{H\in\mathsf{H}\,:\,H\subseteq F\}$ and $H/F:=\{H\setminus F:H\in\mathsf{H}\}.$ In the same paper, the \emph{weighted integer points enumerator} is defined as
\begin{equation}\label{for1}
F_q(Q_\mathsf{H})\;\;\;:=\sum_{\mathcal{F}\in L(Pe^{n-1})}
q^{\mathsf{rk}(\mathsf{H}/\mathcal{F})}M_\mathcal{F},
\end{equation}
where $\mathsf{rk}(\mathsf{H}/\mathcal{F}):=n-c(\mathsf{H}/\mathcal{F})$, $c(\mathsf{H}/\mathcal{F})$ is the number of connected components of the hypergraph $\mathsf{H}/\mathcal{F}$ and $L(Pe^{n-1})$ is the face lattice of the generalized permutohedron $Pe^{n-1}$.

\begin{definition}\emph{ 
For a simple graph $\Gamma=([n],E)$ and an integer $m\in\mathbb{N}$ we define the \emph{$m-$graph polytope} 
$$Q_{\Gamma,m}\;\;\;:=\sum_{\substack{S\subseteq[n],\,|S|\leq m+1\\\Gamma|_S\text{ is connected}}}\Delta_S.$$
}\end{definition}

Note that $Q_{\Gamma,1}$ is a graphical zonotope and $Q_{\Gamma,m}$ is a graph$-$associahedron for $m\geq n-1$.
If $\mathsf{H}_{m}(\Gamma):=\left\{S\subseteq[n]\,:\,|S|\leq m+1\text{ and }\Gamma|_{S}\text{ is connected}\right\}$, then $Q_{\Gamma,m}$ is the hypergraphical polytope $Q_{\mathsf{H}_{m}(\Gamma)}$.

\section{Hopf algebras of decorated graphs}

We say that $\Gamma^w=([n],E,w)$ is a \emph{decorated graph} if $\Gamma=([n],E)$ is a simple graph and $w:E\rightarrow\mathbb{N}$ is the \emph{decoration} of $\Gamma$. Let $\Gamma^w|_S$ be the induced decorated subgraph on $S\subseteq [n]$ and $\Gamma^w/S$ be induced subgraph on $[n]\setminus S$ with additional edges $uv$ for all pairs of vertices $u,v\in[n]\setminus S$ connected by edge
paths through $S$. The decoration of a new edge $uv$ is the minimal sum of decorations over edge paths through $S$, i.e.
$$w(uv)\;\;\;\;\;:=\min_{\substack{us_1,\,s_1s_2,\,\ldots,\,s_kv\in\, E(\Gamma^w)\\ s_1,s_2,\ldots,s_k\in S}}\{w(us_1)+w(s_1s_2)+\cdots+w(s_kv)\}.$$
We call $\Gamma^w|_S$ the \emph{ripping} of a decorated graph $\Gamma^w$ to $S$ and $\Gamma^w/S$ the \emph{sewing} of a decorated graph $\Gamma^w$ by $S$.

\begin{example}
For the decorated graph $\Gamma^w$ and the subset $S=\{1,4,5,6\}$, the ripping $\Gamma^w|_S$ and the sewing $\Gamma^w/S$ are given on the Figure~\ref{fig:slika2}.

\vspace*{-4mm}
\begin{figure}[h!]
\centering
\includegraphics[width=0.65\textwidth]{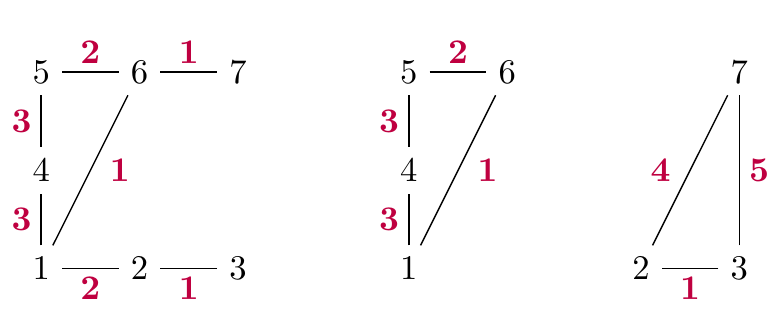}
\caption{Decorated graphs $\Gamma^w$, $\Gamma^w|_S$ and $\Gamma^w/S$.}
\label{fig:slika2}
\end{figure}
\end{example}

We say that $\Gamma^w$ is an \emph{$1-$uniform decorated graph}, denoted by $\Gamma^\bold{1}$, if $w(uv)=1$ for all $uv\in E$. Then, the decoration of an additional edge $uv$ in $\Gamma^{\textbf1}/S$, for $S\subseteq [n]$, is the length of the shortest path through $S$ from $u$ to $v$.

\medskip

Two decorated graphs are \emph{isomorphic} if there is a bijection between them, which preserves decoration.
Let $\mathcal{G}^{W}_n$ denote the $\bold{k}-$span of all isomorphism classes of decorated graphs on $n$ vertices, where  $\mathcal{G}^{W}_0:=\bold{k}\{\emptyset\}$ and $\emptyset$ is the unique decorated graph on the empty set.  For each $m\in\mathbb{N}$ we will endow 
$$\mathcal{G}^{W}:=\bigoplus_{n\geq0}\mathcal{G}^{W}_n$$
with the structure of a graded Hopf algebra. The \emph{unit} $u:\bold{k}\rightarrow\mathcal{G}^{W}$, \emph{counit} $\varepsilon:\mathcal{G}^{W}\rightarrow\bold{k}$ and the \emph{product} $\mu:\mathcal{G}^{W}\otimes\mathcal{G}^{W}\rightarrow\mathcal{G}^{W}$
are the same for all $m$ and they are defined by $u(1):=\emptyset$, 
$$\varepsilon(\Gamma^w)=\begin{cases}1,&\Gamma^w=\emptyset,\\0,&\text{otherwise},\end{cases}
\qquad\text{and}\qquad
\Gamma^{w_1}_1\cdot \Gamma^{w_2}_2:=(\Gamma_1\sqcup \Gamma_2)^{w_1w_2}.$$ 
Here, the decoration $w_1w_2:E_1\sqcup E_2\rightarrow\mathbb{N}$ is defined with
$$w_1w_2(uv)=\begin{cases}w_1(uv),&uv\in E_1,\\w_2(uv),&uv\in E_2. 
\end{cases}$$
For an integer $m\in\mathbb{N}$ we define the \emph{coproduct} $\Delta_m:\mathcal{G}^{W}\rightarrow\mathcal{G}^{W}\otimes\mathcal{G}^{W}$ by
$$\Delta_m(\Gamma^w)=\sum_{S\subseteq [n]}\mathrm{pr}_m\left(\Gamma^w|_S\right)\otimes\mathrm{pr}_m\left(\Gamma^w/S\right),\;\;\;\Gamma^w\in\mathcal{G}^{W}_{n},$$ 
where $\mathrm{pr}_m:\mathcal{G}^{W}\rightarrow\mathcal{G}^{W}$ is the map which deletes all edges greater than $m$ of the decoration $w$.\\

The antipode $\mathcal{S}_m$ of $\Gamma^w$ is determined by general Takeuchi’s formula for the antipode of a graded connected bialgebra
$$\mathcal{S}(\Gamma^w)=\sum_{k\geq1}(-1)^k\sum_{\mathcal{F}_k}\prod_{i=1}^k\mathrm{pr}_m(\Gamma^w|_{F_{i}}/F_{i-1}),$$
where the inner sum goes over all flags of subset $\mathcal{F}_k:\emptyset=F_0\subset F_1\subset\cdots\subset F_k=[n].$
\begin{proposition}
For all $m\in\mathbb{N}$,
$\mathcal{G}^{W,m}=(\mathcal{G}^W,\mu,u,\Delta_m,\varepsilon,\mathcal{S}_m)$ is a graded connected Hopf algebra.
\end{proposition}
\begin{proof}
We prove compatibility of the product and coproduct and coassociativity, since other properties are straightforward. For a decorated graph $\Gamma^w$ of $V$, one has the following identities
{\small
\begin{align*}
((\Delta\otimes\mathsf{Id})\circ\Delta)(\Gamma^w)
&=\sum_{\emptyset\subseteq S_1\subseteq S_2\subseteq [n]}
\mathrm{pr}_m(\Gamma^w|_{S_1})\otimes\mathrm{pr}_m(\Gamma^w|_{S_2}/S_1)\otimes\mathrm{pr}_m(\Gamma^w/S_2)\\
((\mathsf{Id}\otimes\Delta)\circ\Delta)(\Gamma^w)
&=\sum_{\emptyset\subseteq S_1\subseteq S_2\subseteq [n]}
\mathrm{pr}_m(\Gamma^w|_{S_1})\otimes\mathrm{pr}_m(\Gamma^w/S_1|_{S_2\setminus S_1})\\
&\qquad\qquad\qquad\qquad\qquad\;\;\;\;
\otimes\mathrm{pr}_m((\Gamma^w/S_1)/(S_2\setminus S_1)).
\end{align*}}

\vspace*{-3mm}
\noindent
Since $(\Gamma^w/S_1)/(S_2\setminus S_1)=\Gamma^w/S_2$, it is sufficient to show that 
$$\mathrm{pr}_m(\Gamma^w|_{S_2}/S_1)=\mathrm{pr}_m(\Gamma^w/S_1|_{S_2\setminus S_1}).$$ 
Let $uv\in \Gamma^w|_{S_2}/S_1$ and $w(uv)\leq m.$ That means that $u$ and $v$ are connected in $\Gamma^w|_{S_2\setminus S_1}$ or there is a path in $\Gamma^w|_{S_2}$ throught $S_1$. In both cases $uv\in\mathrm{pr}_m(\Gamma^w/S_1|_{S_2\setminus S_1})$. Also, if $uv\in\Gamma^w/S_1|_{S_2\setminus S_1}$ and $w(uv)\leq m$, then $u$ and $v$ are connected in $\Gamma^w|_{S_2\setminus S_1}$ or there is a path in $\Gamma^w|_{S_2}$ throught $S_1.$ Again, in both cases $uv\in\mathrm{pr}_m(\Gamma^w|_{S_2}/S_1).$

\medskip

\noindent
Furthermore, for a pair of decorated graphs $\Gamma^{w_1}_1\in\mathcal{G}^{W,m}_{n_1}$ and $\Gamma^{w_2}_2\in\mathcal{G}^{W,m}_{n_2}$ and subsets $S_1\subset [n_1]$ and $S_2\subset [n_2]$ one has isomorphisms
$$(\Gamma^{w_1}_1|_{S_1})\cdot(\Gamma^{w_2}_2|_{S_2})=(\Gamma^{w_1}_1\cdot \Gamma^{w_2}_2)|_{S_1\sqcup S_2}$$
and
$$(\Gamma^{w_1}_1/S_1)\cdot(\Gamma^{w_2}_2/S_2)=(\Gamma^{w_1}_1\cdot \Gamma^{w_2}_2)/(S_1\sqcup S_2),$$ 
which proves commutativity of the bialgebra diagram in the definition of  Hopf algebra.
\end{proof}

Now we define $\zeta_q:\mathcal{G}^{W,m}\rightarrow\bold{k}[\,q\,]$ by
$$\zeta_q(\Gamma^w):=q^{n-c(\Gamma^w)},\;\;\;\Gamma^w\in\mathcal{G}^{W,m}_n,$$
where $c(\Gamma^w)$ is the number of connected components of $\Gamma.$ It is straightforward that $\zeta_q$ is a multiplicative morphism, which turns $(\mathcal{G}^{W,m},\zeta_q)$ into a combinatorial Hopf algebra.

\medskip

By the fundamental theorem of combinatorial Hopf algebras (see~\cite{ABS}, Theorem 4.1), there is a a unique morphism
$$\Psi_q^m:(\mathcal{G}^{W,m},\zeta_q)\rightarrow(\mathcal{Q}Sym,\zeta).$$
In the monomial basis it is given by
$$\Psi_q^m(\Gamma^w)=\sum_{\alpha\models n}(\zeta_q)_\alpha(\Gamma^w)M_\alpha.$$
For a composition $\alpha=(\alpha_1,\alpha_2,\ldots,\alpha_k)$, the coefficient $(\zeta_q)_\alpha(\Gamma^w)$ is determined by
$$(\zeta_q)_\alpha(\Gamma^w)=\zeta_q^{\otimes k}\circ(p_{\alpha_1}\otimes p_{\alpha_2}\otimes\cdots\otimes p_{\alpha_k})\circ\Delta_m^{k-1}(\Gamma^w),
$$
where $p_i$ is the projection on the $i-$th homogeneous component and $\Delta_m^{k-1}$ is the $(k-1)-$fold coproduct map of $\mathcal{G}^{W,m}$.\\

For a decorated graph $\Gamma^w$ and a flag of subsets $\mathcal{F}:\emptyset=F_0\subset F_1\subset\cdots\subset F_k=[\,n\,]$ let
$$\Gamma^w/\mathcal{F}\;:=\;\bigsqcup_{i=1}^k\Gamma^w|_{F_i}/F_{i-1}.$$
Thus, the coefficient corresponding to a composition $\alpha=(\alpha_1,\alpha_2,\ldots,\alpha_k)\models n$ is a
polynomial in $q$ determined by
\begin{align*}
(\zeta_q)_\alpha(\Gamma^w)
=
\sum_{\mathcal{F}:\mathsf{type}(\mathcal{F})=\alpha}\prod_{i=1}^kq^{|F_i/F_{i-1}|-c(\mathrm{pr}_m(\Gamma^w|_{F_i}/F_{i-1}))}
=
\sum_{\mathcal{F}:\mathsf{type}(\mathcal{F})=\alpha}q^{\mathsf{rk}_m(\Gamma^w/\mathcal{F})},
\end{align*}
where 
\begin{equation*}
\mathsf{rk}_m(\Gamma^w/\mathcal{F}):=n-\sum_{i=1}^kc(\mathrm{pr}_m(\Gamma^w|_{F_i}/F_{i-1})).
\end{equation*}
Finally, we obtain
\begin{equation}\label{for2}
\Psi_q^m(\Gamma^w)=\sum_{\mathcal{F}\in L(Pe^{n-1})}q^{\mathsf{rk}_m(\Gamma^w/\mathcal{F})}M_\mathcal{F},
\end{equation}
where $L(Pe^{n-1})$ is the face lattice of the generalized permutohedron $Pe^{n-1}.$

\begin{theorem}
Given a simple graph $\Gamma$ and $m\in\mathbb{N}$, let $Q_{\Gamma,m}$ be the corresponding $m-$graph polytope. Then, the following identity holds
$$F_q(Q_{\Gamma,m})=\Psi^m_q(\Gamma^{\bf 1}),$$
where $F_q(Q_{\Gamma,m})$ is the weighted integer points enumerator of a $m-$graph polytope $Q_{\Gamma,m}$.
\end{theorem}
\begin{proof}
From~(\ref{for1}) and~(\ref{for2}), it is suifficient to prove that for an any flag of subsets $\mathcal{F}:\emptyset=F_0\subset F_1\subset\cdots\subset F_k=[n]$ and $1\leq i\leq k$ it holds
$$c(\mathrm{pr}_m(\Gamma^{\bold{1}}|_{F_i}/F_{i-1}))=c(\mathsf{H}_{m}(\Gamma)|_{F_i}/F_{i-1}).$$
All edges in the decorated graph $\mathrm{pr}_m(\Gamma^{\textbf1}|_{F_i}/F_{i-1})$ have decorations less than or equal to $m$, i.e. $uv\in E\left(\mathrm{pr}_m(\Gamma^{\textbf1}|_{F_i}/F_{i-1})\right)$ if and only if there is a path from $u$ to $v$ of the length not greater than $m$ in the graph $\Gamma|_{F_i}$. Equivalently, there exists $H\in\mathsf{H}_{m}(\Gamma)$ such that $u,v\in H.$
\end{proof}

From the general theorem of generalized permutohedra (see~\cite{GPS}, Theorem 4.4), the $f-$polynomial of a $m-$graph polytope $Q_{\Gamma,m}$ is determined by the principal specialization of the enumerator $F_q(Q_{\Gamma,m})$, i.e.
\begin{align}
\label{f_vektor}
f(Q_{\Gamma,m},q)=(-1)^n\mathsf{ps}^1(\Psi_{-q}^m(\Gamma^\bold{1}))(-1).
\end{align}

\begin{example}
For the line graph $L_4=([4],\{12,23,34\})$, we have
\begin{align*}
F_q(Q_{L_4,1})
&=q^3M_4+2q^2(M_{1,3}+M_{3,1}+M_{2,2})\\
&+2q(M_{1,3}+M_{3,1}+M_{2,2}+3M_{2,1,1}+3M_{1,2,1}+3M_{1,1,2})\\
&+2M_{2,2}+6M_{1,1,2}+6M_{1,2,1}+6M_{2,1,1}+24M_{1,1,1,1},\\
F_q(Q_{L_4,2})
&=q^3M_4+2q^2(M_{1,3}+2M_{3,1}+M_{2,2})\\
&+2q(M_{3,1}+2M_{2,2}+5M_{1,1,2}+4M_{1,2,1}+3M_{2,1,1})\\
&+\;\;\;\;\;\;\;\;\;\;\;\;\;\,2M_{1,1,2}+4M_{1,2,1}+6M_{2,1,1}+24M_{1,1,1,1},\\
F_q(Q_{L_4,3})
&=q^3M_4+q^2(2M_{3,1}+4M_{3,1}+3M_{2,2})\\
&=q(
2M_{3,1}+M_{2,2}+12M_{1,1,2}+8M_{1,2,1}+6M_{2,1,1})\\
&+\;\;\;\;\;\;\;\;\;\;\;\;\;\;\;\;\;\;\;\;\;\;\;\;\;\;\;\;\;4M_{1,2,1}+6M_{2,1,1}+24M_{1,1,1,1}.
\end{align*} 
Corresponding $f$--polynomials are determined by the principal specialization

\vspace*{-9mm}
\begin{align*}
f(Q_{L_4,1},q)
&=q^3+6q^2+12q+8,\\
f(Q_{L_4,2},q)
&=q^3+8q^2+18q+12,\\
f(Q_{L_4,3},q)
&=q^3+9q^2+21q+14.
\end{align*}

\begin{figure}[h!]
\centering
\includegraphics[width=0.75\textwidth]{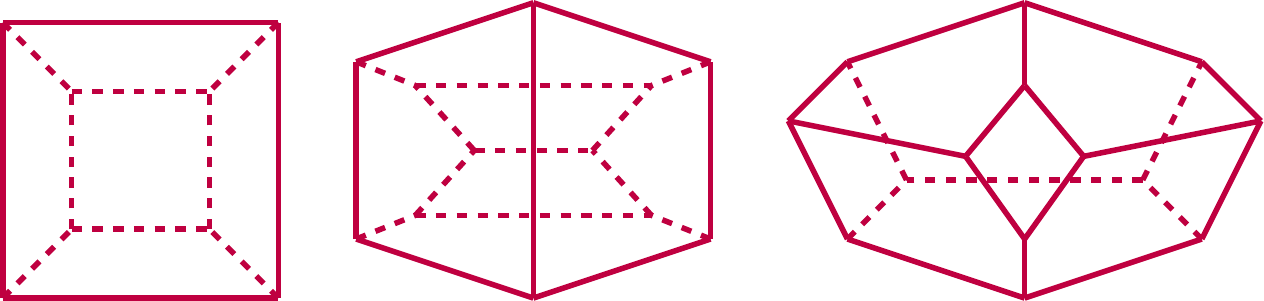}
\caption{Graph polytopes $Q_{L_4,1}$, $Q_{L_4,2}$ and $Q_{L_4,3}$.}
\label{fig:slika1}
\end{figure}

\end{example}

In the sequel, we associate the sequence of quasysimmetric functions 
$$(\Psi_q^1(\Gamma^\bold{1}),\Psi_q^2(\Gamma^\bold{1}),\ldots,\Psi_q^{n-1}(\Gamma^\bold{1}))$$ 
to a simple graph $\Gamma$ on $n$ vertices. 
The following theorem answers how this sequence stabilizes. 

\begin{theorem}\label{thm_1}
Given a connected simple graph $\Gamma$, let $m$ be the cardinality of the maximal subset $M\subseteq[n]$ such that $\Gamma|_M\cong L_{|M|}$. Then, for all $k\geq |M|$ we have
$$\Psi^{k}_q(\Gamma^{\bold{1}})=\Psi^{|M|-1}_q(\Gamma^{\bold{1}}).$$
\end{theorem}
\begin{proof}
It is sufficient to prove that $c\left(\mathrm{pr}_k(\Gamma^\bold{1}/\mathcal{F})\right)=c\left(\mathrm{pr}_{|M|-1}(\Gamma^\bold{1}/\mathcal{F})\right)$ for all $k\geq|M|$ and for all flags $\mathcal{F}$ of subsets of the set $[n]$.
Let us suppose that $u,v\in F_{i}\setminus F_{i-1}$ for the flag $\mathcal{F}:\emptyset=F_0\subset F_1\subset\cdots\subset F_k=[n]$. 
Since the maximal distance between vertices in the graph $\Gamma$ is $m-1$, it follows that $w(u,v)\leq |M|-1$. 
If $u$ and $v$ are connected in $\Gamma|_{F_{i}}$, then $u$ and $v$ are not connected in $\mathrm{pr}_t(\Gamma^\bold{1}/\mathcal{F})$ for $t<w(u,v)$, but they are connected in $\mathrm{pr}_t(\Gamma^\bold{1}/\mathcal{F})$ for all $t\geq w(u,v)\geq|M|-1$. Similarly, if $u$ and $v$ are not connected in $\Gamma|_{F_{i}}$, then $u$ and $v$ are not connected in $\mathrm{pr}_k(\Gamma^\bold{1}/\mathcal{F})$ for any $k$.
\end{proof}

\begin{example}
Note that $\Psi^k_q(K_n^\bold{1})=\Psi^1_q(K_n^\bold{1})$ for all $k\geq1$, where $K_n$ is the complete graph on $n$ vertices. For the star graph $S_n$, it holds $\Psi^k_q(S_n^\bold{1})=\Psi^2_q(S_n^\bold{1})$ for all $k\geq2$.
\end{example}

\begin{corollary}
The line graph $L_n$ is the only graph on $n$ vertices such that $(\Psi^1_q(L_n^\bold{1}),$ $\Psi^2_q(L_n^\bold{1}),$ $\ldots,$$ \Psi^{n-1}_q(L_n^\bold{1}))$ 
are different quasysimmetric functions.
\end{corollary}

For the graph $\Gamma$ on $n$ vertices, the enumerator $\Psi^1_0(\Gamma^\bold{1})$ is Stanley chromatic symmetric function of graph $\Gamma$ and the enumerator $\Psi^{n-1}_0(\Gamma^\bold{1})$ is chromatic quasisymmetric function introduced in \cite{G}. There is only one pair of graphs on five vertices with the same Stanley chromatic functions, but their chromatic quasisymmetric functions are different. On the other hand, there are three pairs of graphs on six vertices whose chromatic quasisymmetric functions are the same, but the corresponding Stanley chromatic functions are not. 

\begin{example}
For graphs $\Gamma_1$ and $\Gamma_2$, see Figure \ref{fig:slika11}, with the same Stanley chromatic symmetric functions, we have
{\small
\begin{align*}
\Psi^1_0(\Gamma_1)=\Psi^1_0(\Gamma_2)&=4M_{1,2,2}+4M_{2,1,2}+4M_{2,2,1}\\
&+24M_{1,1,1,2}+24M_{1,1,2,1}+24M_{1,2,1,1}+24M_{2,1,1,1}+120M_{1,1,1,1,1}.
\end{align*}
}
\begin{figure}[h!]
\centering
\includegraphics[width=0.35\textwidth]{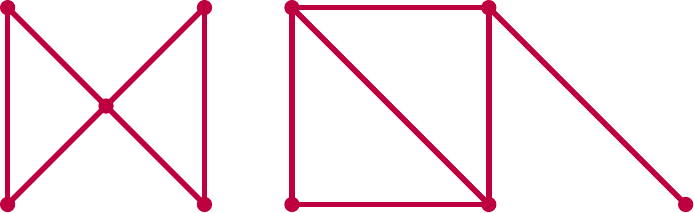}
\caption{Graphs $\Gamma_1$ and $\Gamma_2$ with the same Stanley chromatic symmetric functions}
\label{fig:slika11}
\end{figure}

\vspace*{-4mm}
\noindent
By Theorem \ref{thm_1}, $\Psi_0^k(\Gamma_1)=\Psi_0^2(\Gamma_1)$, for $k\geq2$, and $\Psi_0^k(\Gamma_2)=\Psi_0^3(\Gamma_2)$, for $k\geq3$. Also, the following equations hold
{\small
\begin{align*}
\Psi^2_0(\Gamma_1)&=4M_{2,2,1}\qquad\qquad\;\;\;\;
+\;8\,M_{1,1,2,1}+16M_{1,2,1,1}+24M_{2,1,1,1}+120M_{1,1,1,1,1},\\
\Psi^2_0(\Gamma_2)&=\qquad\qquad\;\, 6M_{1,1,1,2}+10M_{1,1,2,1}+16M_{1,2,1,1}+24M_{2,1,1,1}+120M_{1,1,1,1,1},
\\
\Psi^3_0(\Gamma_2)&=\qquad\qquad\qquad\qquad\;\;\;\;\;\;\; 6M_{1,1,2,1}+16M_{1,2,1,1}+24M_{2,1,1,1}+120M_{1,1,1,1,1}.
\end{align*}}
\end{example}

\begin{example}
For the graphs in Figure 4, we have that the corresponding chromatic quasisymmetric functions are the same, i.e.
$\Psi^3_0(\Gamma_3)=\Psi^3_0(\Gamma_4)$. Note that the coefficient on $M_{1,1,1,1,2}$ in $\Psi^2_0(\Gamma_3)$ is 24 and the coefficient on $M_{1,1,1,1,2}$ in $\Psi^2_0(\Gamma_4)$ is 0, so
$$\Psi^2_0(\Gamma_3)\neq\Psi^2_0(\Gamma_4).$$ 

\vspace*{-4mm}
\begin{figure}[h!]
\centering
\includegraphics[width=0.35\textwidth]{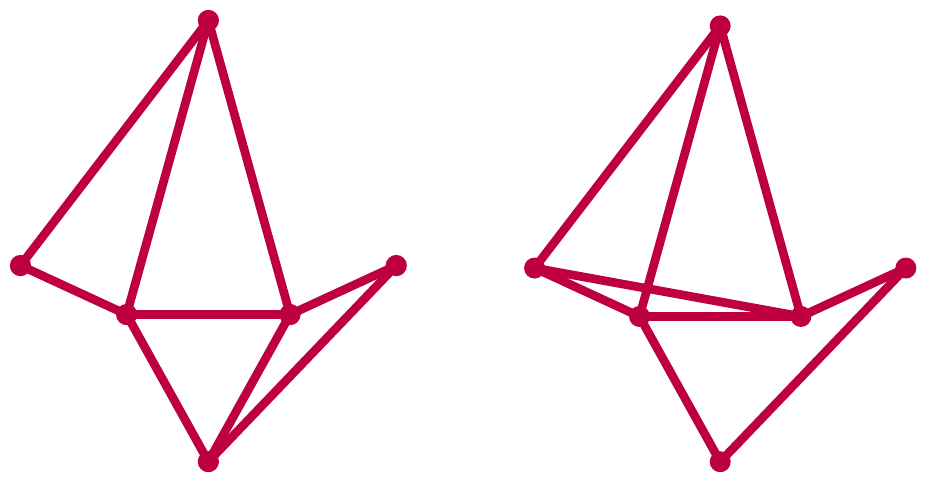}
\caption{Graphs $\Gamma_3$ and $\Gamma_4$ with the same chromatic quasisymmetric functions}
\label{fig:slika1}
\end{figure}

\noindent
Moreover, in \cite{GPS} is shown that $\Psi_0^1(\Gamma_3)\neq\Psi_0^1(\Gamma_4).$
\end{example}

In the previous theorem we have shown how the quasysimmetric functions $(\Psi_q^1(\Gamma^\bold{1}),$ $\Psi_q^2(\Gamma^\bold{1}),$ $\ldots,$ $\Psi_q^{n-1}(\Gamma^\bold{1}))$ 
associated to the graph $\Gamma$ are stabilized. Now, the following question arises: What happens with the sequence of polytopes $(Q_{\Gamma,1}, Q_{\Gamma,2},\ldots, Q_{\Gamma,n-1})$ associated to the graph $\Gamma.$ To answer this question, we will need the following lemma.

\begin{lemma}[\cite{DP}, Lemma 2.4]\label{lema_1}
Consider two polytopes $P$ and $Q$. Let $\psi$ be an injection from the
vertex set of $P$ to the vertex set of $P+Q$ such that, for every vertex $u$ of $P$, $\psi(u)=u+v$, where $v$ is a vertex of $Q$. If $\psi$ is a bijection, then the normal fan of $P$ coincides with the normal fan of $P+Q$.
\end{lemma}

\begin{proposition}
Given a connected simple graph $\Gamma$, let $m$ be the cardinality of the maximal subset $M\subseteq[n]$ such that $\Gamma|_M\cong L_{|M|}$. Then, for all $k\geq|M|$ the polytopes $Q_{\Gamma,k}$ and $Q_{\Gamma,|M|-1}$ are normally equivalent, i.e. the normal fan of $Q_{\Gamma,k}$ coincides with the normal fan of $Q_{\Gamma,|M|-1}.$
\end{proposition}
\begin{proof}
From Theorem \ref{thm_1} and (\ref{f_vektor}), we have that for all $k\geq |M|$  polytopes $Q_{\Gamma,k}$ and $Q_{\Gamma,|M|-1}$ have the same $f-$polynomial. Specially, for all $k\geq |M|$ polytopes $Q_{\Gamma,k}$ and $Q_{\Gamma,|M|-1}$ have the same number of vertices. Since
$$Q_{\Gamma,k}\;\;\;=\;\;\;Q_{\Gamma,|M|-1}\;\;\;+\sum_{\substack{S\subseteq[n],|M|+1\leq|S|\leq k+1\\ \Gamma|_S\text{ is connected}}}\Delta_S\;\;=\;\;Q_{\Gamma,|M|-1}\;\;+\;\;P,$$
it holds that the map $\psi$, from the previous lemma, is a bijection from the vertex set of $Q_{\Gamma,|M|-1}$ to the vertex set of $Q_{\Gamma,|M|-1}+P.$ According to Lemma \ref{lema_1}, the normal fan of the polytope $Q_{\Gamma,|M|-1}$ coincides with the normal fan of the polytope $Q_{\Gamma,|M|-1}+P$.
\end{proof}

\begin{theorem}
For a connected simple graph $\Gamma$ on $n$ vertices and $m\in\mathbb{N}$, the polytopes 
$$Q_{\Gamma,m}\;\;\;:=\sum_{\substack{S\subseteq[n],\,|S|\leq m+1\\c(\Gamma|_S)=1}}\Delta_S
\qquad\text{and}\qquad
Q^L_{\Gamma,m}\;\;\;:=\sum_{\substack{S\subseteq[n],\,|S|\leq m+1\\\Gamma|_S\cong L_{|S|}}}\Delta_S.$$
are normally equivalent.
\end{theorem}
\begin{proof}
Let be $M\subseteq[n]$ such that $|M|=m$ and $\Gamma|_M\cong L_m.$ If all $S\subseteq[n]$, $|S|\leq m$, satisfy $\Gamma|_S\cong L_{|S|}$, the statment is true. Further, let $K$ be a subset of maximum cardinality $k$ of the nonempty set
$$\{S\subseteq[n]\,:\,|S|\leq m,\,\Gamma|_S\ncong L_{|S|}\}.$$
It is sufficient to prove that polytopes $Q_{\Gamma,m}$ and $Q_{\Gamma,m}-\Delta_K$ are normally equivalent. By the previous theorem, there exists $k'<k$ such that polytopes $Q_{\Gamma|_K,k}$ and $Q_{\Gamma|_K,k'}$ are normally equivalent. In particular, it means that 
$$Q_{\Gamma|_K,k}=Q_{\Gamma|_K,k-1}+\Delta_K
\qquad\text{and}\qquad 
Q_{\Gamma|_K,k-1}$$
are normally equivalent polytopes. Since $Q_{\Gamma|_K,k}$ is the Minkowski summand of $Q_{\Gamma,m}$, it implies that polytopes
$Q_{\Gamma,m}$ and $Q_{\Gamma,m}-\Delta_{K}$
are normally equivalent as well.
\end{proof}

\begin{corollary}
For a connected simple graph $\Gamma$ on $n$ vertices, the polytopes 
$$Q_{\Gamma}\;\;\;:=\sum_{\substack{S\subseteq[n],\\c(\Gamma|_S)=1}}\Delta_S
\qquad\text{and}\qquad
Q^L_{\Gamma}\;\;\;:=\sum_{\substack{S\subseteq[n],\\\Gamma|_S\cong L_{|S|}}}\Delta_S.$$
are normally equivalent.
\end{corollary}

\section{$\mathcal{H}-$posets}

For a given vertex $v$ of a $m-$graph polytope $Q_{\Gamma,m}$ there exists a poset $\mathsf{P}_v$ (described in \cite{PRW}, Corollary 3.9) whose linear extensions corresponding to the \emph{Weyl chambers} are contained in the normal cone of the vertex $v$. 
For example, if $\mathsf{P}_v:(1<2,3<2)$ then the normal cone of the vertex $v$ contains Weyl chambers determined by $x_1\leq x_3\leq x_2$ and $x_3\leq x_1\leq x_2.$

\medskip

We can regard a poset $\mathsf{P}$ 
as a directed graph where $j<_{\mathsf{P}}i$ 
if and only if there is a directed path from $i$ to $j$ in that directed graph. In the general case, vice versa does not hold. If a directed graph is \emph{acyclic}, i.e. there exist no vertices $v_1,v_2,\cdots,v_k$ such that 
$v_1\rightarrow v_2\rightarrow\cdots\rightarrow v_l\rightarrow v_1$, we can view this directed graph as a binary relation whose transitive closure defines a poset.

\medskip
 
For a given hypergraph $\mathsf{H}$ on the set $[n]$, an $1-$\emph{orientation of a hyperedge} $H\in\mathsf{H}$ is a par $(h_1,H_2)$, where $h_1$ is a distinguished element of $H\subseteq[n]$, and $H_2=H\setminus\{h_1\}$. An \emph{$1-$orientation $\mathcal{O}$ of the hypegraph} $\mathsf{H}$ is the set of $1-$orientations of all its hyperedges. Then we can construct an oriented multigraph $\mathsf{H}/\mathcal{O}$ on the set $[n]$ with $h_1\rightarrow h_2$ for all $h_2\in H_2$ satisfying $(h_1,H_2)\in\mathcal{O}$. Specially, if the oriented multigraph $\mathsf{H}/\mathcal{O}$ has no cycles we say that the $1-$orientation $\mathcal{O}$ is \emph{acyclic}. In that case $\mathsf{H}/\mathcal{O}$ is an oriented graph. 

\begin{example}
For the hypergraph $\mathsf{H}_2(L_4)=([4],\{12,23,34,123,234\})$ the $1$--orientation 
$$\mathcal{O}_1=\{(1,\{2\}), (2,\{3\}), (3,\{4\}),(1,\{2,3\}),(2,\{3,4\}\}$$ 
is acyclic. The $1$--orientation 
$$\mathcal{O}_2=\{(1,\{2\}), (2,\{3\}), (3,\{4\}),(3,\{1,2\}),(4,\{2,3\}\}$$ 
is not acyclic, since $\mathsf{H}/\mathcal{O}_2$ has cycle $2\to3\to4\to2$. 
\end{example}

Let $\mathcal{O}$ be an acyclic $1-$orientation of a hypegraph $\mathsf{H}$ on the set $[n]$. The transitive closure of the acyclic orientied graph $\mathsf{H}/\mathcal{O}$ is a poset $\mathsf{P}$ such that for all $H\in\mathsf{H}$ the restriction $\mathsf{P}|_H$ is a poset whose Hasse diagram is a \emph{rooted tree}. The root of this rooted tree is the first component of $1-$orientation $(h_1,H_2)$ of a hyperedge $H$ in $\mathsf{H}.$

\begin{example}
The $1$--orientation $\mathcal{O}_2$ from the pervious example is acyclic, so the transitive closure od the orientied graph $\mathsf{H}/\mathcal{O}_2$ is the poset $\mathsf{P}(2<1,3<1,4<1,3<2,4<2,4<3)$.
\end{example}

\begin{definition}\label{Hposet}\emph{
A poset $\mathsf{P}$ on the set $[n]$ is a \emph{$\mathcal{H}-$poset} of a hypergraph $\mathsf{H}$ if
\begin{enumerate}
\item for all $H\in\mathsf{H}$ the Hasse diagram of the restriction $\mathsf{P}|_H$ is a rooted tree, 
\item $i\lessdot_\mathsf{P}j$ if and only if there exists $H\in\mathsf{H}$ such that $\{i,j\}\subseteq H$, where $i\lessdot_\mathsf{P}j$ means that there is no $k\in\mathsf{P}$ such that $i<_\mathsf{P}k<_\mathsf{P}j.$
\end{enumerate}}
\end{definition}

Note that $\mathcal{B}-$trees defined in \cite{P} satisfy the requirements of the previous definition, so $\mathcal{B}-$trees are the special kind of $\mathcal{H}-$posets.

\begin{example}
Let $L_4$ be the line graph on the set $[4]$  with edges $\{12,23,34\}$. For the hypergraph $\mathsf{H}_1(L_4)$ there are eight  $\mathcal{H}-$posets: 
\[\mathsf{P}_{1},\mathsf{P}_{2},\mathsf{P}_{3},\mathsf{P}_{4},\mathsf{P}_{5},\mathsf{P}_{6},\mathsf{P}_{7},\mathsf{P}_{8}.\]
Also, there are twelve $\mathcal{H}-$posets for the hypergraph $\mathsf{H}_2(L_4)$:
\[\mathsf{P}_{1},\mathsf{P}_{2},\mathsf{P}_{3},\mathsf{P}_{4},\mathsf{P}_{9},\mathsf{P}_{10},\mathsf{P}_{11},\mathsf{P}_{12},\mathsf{P}_{13},\mathsf{P}_{14},\mathsf{P}_{15},\mathsf{P}_{16}.\]
Finally, for the hypergraph $\mathsf{H}_3(L_4)$ we have fourteen $\mathcal{H}-$posets:
\[\mathsf{P}_{1},\mathsf{P}_{2},\mathsf{P}_{3},\mathsf{P}_{4},\mathsf{P}_{9},\mathsf{P}_{10},\mathsf{P}_{11},\mathsf{P}_{12},\mathsf{P}_{13},\mathsf{P}_{14},\mathsf{P}_{17},\mathsf{P}_{18},\mathsf{P}_{19},\mathsf{P}_{20}.\]

\vspace*{-3mm}
\begin{figure}[h!]
\centering
\includegraphics[width=0.9\textwidth]{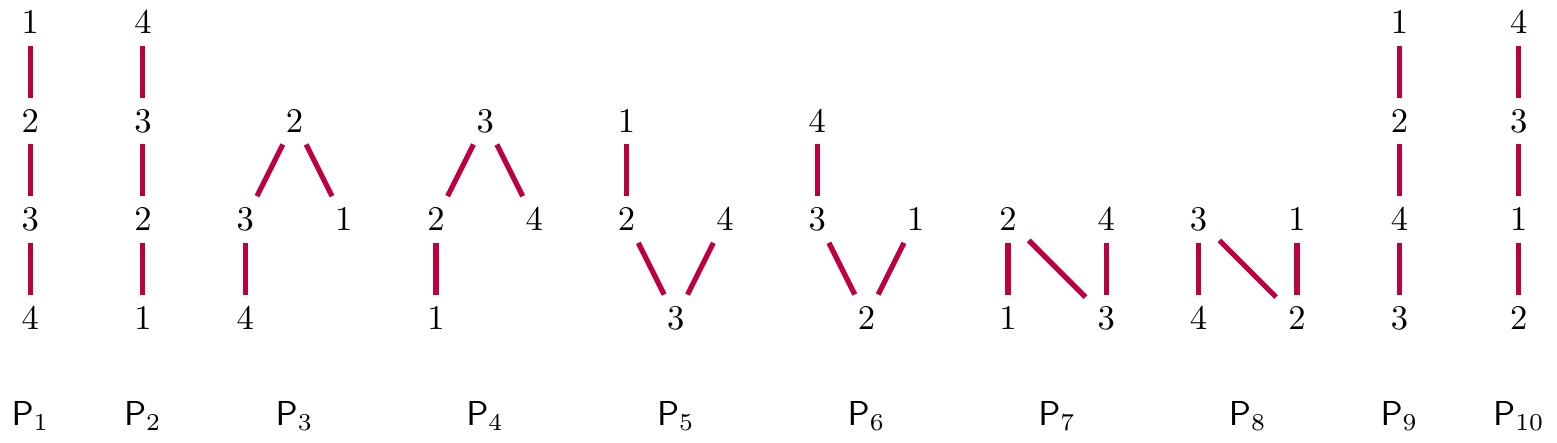}
\includegraphics[width=0.9\textwidth]{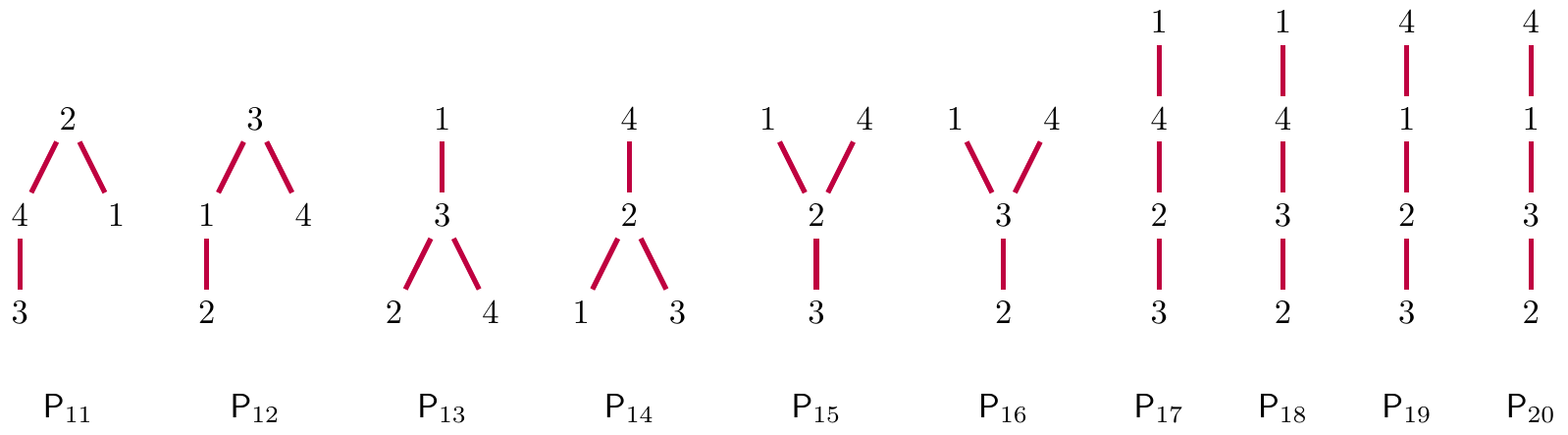}
\caption{The Hasse diagrams of $\mathcal{H}-$posets for $\mathsf{H}_1(L_4), \mathsf{H}_2(L_4)$ and $\mathsf{H}_3(L_4)$}
\label{fig:slika1}
\end{figure}
\end{example}

\begin{proposition}
A poset $\mathsf{P}$ is a $\mathcal{H}-$poset of a hypergraph $\mathsf{H}$ if and only if there exists an acyclic $1-$orientation $\mathcal{O}$ of $\mathsf{H}$ such that $\mathsf{P}$ is the transitive closure of $\mathsf{H}/\mathcal{O}$.
\end{proposition}
\begin{proof}
If $\mathsf{P}$ is a $\mathcal{H}-$poset then for every hyperedge $H\in\mathsf{H}$ there is $h_1\in H$ such that $\mathsf{P}|_H$ is the rooted tree with the root $h_1$. Note that $(h_1,H\setminus\{h_1\})$ is an $1-$orientation of a hyperedge $H$ and that the set of all $1-$orientations of hyperedges forms an acyclic $1-$orientation $\mathcal{O}$ of the hypergraph $\mathsf{H}$, for if $\mathcal{O}$ is not acyclic $1-$orientation, there exists a path $v_1\rightarrow v_2\rightarrow \cdots\rightarrow v_k\rightarrow v_{k+1}=v_1$ in the directed graph $\mathsf{H}/\mathcal{O}$. Then, we have hyperedges $H_1,H_2,\ldots,H_k$ with $1-$orientations $(v_1,H_{1}\setminus\{v_1\}),(v_2,H_{2}\setminus\{v_2\}),\ldots,(v_k,H_{k}\setminus\{v_k\})$ where $v_{i+1}\in H_{i}\setminus\{v_i\}$ for $1\leq i\leq k.$ Since $1-$orientations of hyperedges arise from the poset $\mathsf{P}$, it holds $v_1>_\mathsf{P}v_2>_\mathsf{P}\cdots>_\mathsf{P}v_k>_\mathsf{P}v_1$, the contradiction.
On the other hand, if $\mathcal{O}$ is an acyclic $1-$orientation of $\mathsf{H}$, the transitive closure of $\mathsf{H}/\mathcal{O}$ satisfies the requirements of the definition $\ref{Hposet}$, so $\mathsf{H}/\mathcal{O}$ is a $\mathcal{H}-$poset.
\end{proof}

Let $f:[n]\rightarrow\mathbb{N}$ be a function on the set of vertices of a poset $\mathsf{P}$ on $[n]$. We say that the function $f$ is a \emph{natural $\mathsf{P}-$partition} if 
$f(i)\leq f(j)$ for $v_i\leq_\mathsf{P}v_j$ and a \emph{strict $\mathsf{P}-$partition} if, additionaly, $f(i)<f(j)$ for $i<_\mathsf{P}j$. Denote by $\mathcal{A}(\mathsf{P})$ the set of all natural $\mathsf{P}$-partitions and by $\mathcal{A}_0(\mathsf{P})$ the set of strict $\mathsf{P}-$partitions. Let $F(\mathsf{P})$ be the \emph{quasisymmetric enumerator of strict $\mathsf{P}-$partitions} defined by
\[F(\mathsf{P})=\sum_{f\in\mathcal{A}_0(\mathsf{P})}x_{f(1)}x_{f(2)}\cdots x_{f(n)}.\]

\begin{proposition}
For a simple connected graph $\Gamma$ on the set $[n]$ and $m\geq1$ it holds
\begin{align}\label{poset}
\Psi^m_0(\Gamma)\;\;\;=
\sum_{\mathsf{P}\in\mathcal{H}(\mathsf{H}_m(\Gamma))}F(\mathsf{P}),
\end{align}
where $F$ is the quasisymmetric enumerator of strict $\mathsf{P}-$partitions.
\end{proposition}
\begin{proof}
In \cite{BBM}, Theorem 12, is shown that vertices of an $m-$graph polytope $Q_{\Gamma,m}$ are naturally labeled by acyclic $1-$orientations of the hypergraph $\mathsf{H}_m(\Gamma)$, i.e. that the cone $C_\mathcal{O}$ defined by $x_i\geq x_j$ for $x_i\rightarrow x_j$ in $\mathsf{H}_m(\Gamma)/\mathcal{O}$ is the cone of some vertex in the hypergraphical polytope $Q_{\Gamma,m}.$ If $x_i\geq x_j$ and $x_j\geq x_k$ then $x_i\geq x_k$, so the $\mathcal{H}-$poset $\mathsf{P}$ which is the transitive closure of $\mathsf{H}_m(\Gamma)/\mathcal{O}$ induces the same cone $C_\mathcal{O}.$ Since $\Psi^m_0(\Gamma)$ counts points in the normal cones of vertices of the $m-$graph polytope $Q_{\Gamma,m}$, the equation $(\ref{poset})$ is true.
\end{proof}

In the end we will describe $\mathcal{H}-$posets corresponding to the sequence of polytopes $(Q_{\Gamma,1},Q_{\Gamma,2},\ldots,Q_{\Gamma,n-1})$ associated to a simple graph $\Gamma$. The following theorem shows that from $\mathcal{H}-$posets of $\mathsf{H}_k(\Gamma)$ we can obtain $\mathcal{H}-$posets of the hypergraph $\mathsf{H}_{k+1}(\Gamma).$ Note thar $\mathcal{H}-$posets of $\mathsf{H}_1(\Gamma)$ are transitive closure of acyclic orientations of graph $\Gamma$, and $\mathcal{H}-$posets of $\mathsf{H}_{n-1}(\Gamma)$ are $\mathcal{B}-$trees.

\begin{theorem}
Let $\mathsf{P}$ be a $\mathcal{H}-$poset of the hypergraph $\mathsf{H}_k(\Gamma)$ with the property that $\mathsf{P}$ is not a $\mathcal{H}-$poset of $\mathsf{H}_{k+1}(\Gamma)$. Then there exists an algorithm that creates $\mathcal{H}-$posets of $\mathsf{H}_{k+1}(\Gamma)$ by adding some relation in the $\mathcal{H}-$poset $\mathsf{P}$.
\end{theorem}
\begin{proof}
Let $\mathsf{P}$ be a $\mathcal{H}-$poset of $\mathsf{H}_k(\Gamma)$. It follows that for all subsets $H\in\mathsf{H}_k(\Gamma)$ the restriction $\mathsf{P}|_{H}$ is a rooted tree. As $\mathsf{P}$ is not a $\mathcal{H}-$poset of $\mathsf{H}_{k+1}(\Gamma)$, we can find $H_1\in\mathsf{H}_{k+1}(\Gamma)$ such that $|H_1|=k+1$ and $\mathsf{P}|_{H_1}$ is not a rooted tree. Since there exist at least two subsets of $H_1$ of the cardinality $k$ which belong to $\mathsf{H}_k(\Gamma)$, there are $v_1,v_2\in H_1$ for which there is no element $v_3\in H_1\setminus\{v_1,v_2\}$ satisfying either $v_1<_{\mathsf{P}|_{H_1}}v_3$ or $v_2<_{\mathsf{P}|_{H_1}}v_3$. By adding, respectivelly, the relations $v_1<v_2$ and $v_2<v_1$ to the poset $\mathsf{P}$ we obtain new binary relations on the set $[n]$ whose transitive closures are new posets $\mathsf{P}_1$ are $\mathsf{P}_2$ such that $\mathsf{P}_1|_{H_1}$ and $\mathsf{P}_2|_{H_1}$ are rooted trees. If $\mathsf{P}_1$ and $\mathsf{P}_2$ are $\mathcal{H}-$posets of the hypergraph $\mathsf{H}_{k+1}(\Gamma)$ we will stop this procedure. Otherwise, we will continue in the same way, by taking a new hyperedge $H_2\in\mathsf{H}_{k+1}(\Gamma)$ of the cardinality $k+1$ such that $\mathsf{P}_1|_{H_2}$ or $\mathsf{P}_2|_{H_2}$ is not a rooted tree and by repeating the previously described procedure. This process will end, since there is a finite number of subsets of the cardinality $k+1$ in the hypergraph $\mathsf{H}_{k+1}(\Gamma).$
\end{proof}

\begin{example}
Consider posets $\mathsf{P}_6$, $\mathsf{P}_{10}$, $\mathsf{P}_{16}$, $\mathsf{P}_{18}$ and $\mathsf{P}_{20}$ from the previous example and recall that $\mathsf{P}_6$ is the $\mathcal{H}$--poset for $\mathsf{H}_1(L_4)$, $\mathsf{P}_{10}$ and $\mathsf{P}_{16}$ are the $\mathcal{H}-$posets for $\mathsf{H}_2(L_4)$, $\mathsf{P}_{18}$ and $\mathsf{P}_{20}$ for $\mathsf{H}_3(L_4)$.
From the $\mathcal{H}-$poset $\mathsf{P}_{6}$ we can obtain posets $\mathsf{P}_{10}$ and $\mathsf{P}_{16}$ by adding respectivelly relations $1<3$ and $3<1$ in $\mathsf{P}_{6}$ since $\mathsf{P}_6|_{\{1,2,3\}}$ is not a rooted tree. Simillary, from $\mathsf{P}_{16}$ we can obtain $\mathcal{H}-$posets $\mathsf{P}_{18}$ and $\mathsf{P}_{20}$ by adding respectivelly relations $4<1$ and $1<4$ since $\mathsf{P}_{16}|_{\{1,2,3,4\}}$ is not a rooted tree.
\end{example}


\end{document}